\documentclass[a4paper]{article}
\usepackage{amsthm,amsmath,amssymb}
\usepackage{enumerate,enumitem}
\usepackage{graphicx}
\usepackage{xfrac}

\usepackage{tabularx,booktabs,multirow}
\newcolumntype{C}{>{\centering\arraybackslash}X}
\newcolumntype{D}{>{\centering\arraybackslash}X}

\DeclareGraphicsExtensions{.pdf,.png,.eps}
\graphicspath{{figs/}}

\newtheorem{theorem}{Theorem}
\newtheorem{lemma}[theorem]{Lemma}

\newtheorem{proposition}[theorem]{Proposition}
\newtheorem{corollary}[theorem]{Corollary}
\newtheorem{definition}[theorem]{Definition}

\newtheorem{claim}{Claim}
\newtheorem*{claim*}{Claim}

\theoremstyle{remark}

\newcommand{\arb}{\ensuremath{\operatorname{a}}}
\newcommand{\iarb}{\ensuremath{\operatorname{ia}}}
\newcommand{\wiarb}{\ensuremath{\operatorname{wia}}}
\newcommand{\sarb}{\ensuremath{\operatorname{sa}}}
\newcommand{\isarb}{\ensuremath{\operatorname{isa}}}
\newcommand{\wisarb}{\ensuremath{\operatorname{wisa}}}

\newcommand{\chiacyc}{\ensuremath{\operatorname{\chi_{\rm acyc}}}}
\newcommand{\sd}{\ensuremath{\operatorname{sd}}}
\newcommand{\SD}{\ensuremath{\operatorname{SD}}}

\newcommand{\half}{\ensuremath{\sfrac{1}{2}}}

\newcommand{\cA}{\ensuremath{\mathcal{A}}}
\newcommand{\cD}{\ensuremath{\mathcal{D}}}
\newcommand{\cF}{\ensuremath{\mathcal{F}}}
\newcommand{\cP}{\ensuremath{\mathcal{P}}}
\newcommand{\cS}{\ensuremath{\mathcal{S}}}
\newcommand{\cO}{\ensuremath{\mathcal{O}}}
\newcommand{\cT}{\ensuremath{\mathcal{T}}}
\newcommand{\N}{\ensuremath{\mathbb{N}}}

\begin{document}

\title{Induced and Weak Induced Arboricities}

\author{Maria Axenovich, Philip D\"orr, Jonathan Rollin, Torsten Ueckerdt}

\maketitle

\begin{abstract}
 We define the induced arboricity of a graph $G$, denoted by $\iarb(G)$, as the smallest $k$ such that the edges of $G$ can be covered with $k$ induced forests in $G$.
 This notion generalizes the classical notions of the arboricity and strong chromatic index.
 
 For a class $\cF$ of graphs and a graph parameter $p$, let $p(\cF) = \sup\{p(G) \mid G\in \cF\}$. 
 We show that $\iarb(\cF)$ is bounded from above by an absolute constant depending only on $\cF$, that is $\iarb(\cF)\neq\infty$ if and only if $\chi(\cF \nabla \half) \neq\infty$, where $\cF \nabla \half$ is the class of $\half$-shallow minors of graphs from $\cF$ and $\chi$ is the chromatic number.
 
 Further, we give bounds on $\iarb(\cF)$ when $\cF$ is the class of planar graphs, the class of $d$-degenerate graphs, or the class of graphs having tree-width at most~$d$.
 Specifically, we show that if $\cF$ is the class of planar graphs, then $8 \leq \iarb(\cF) \leq 10$.
 
 In addition, we establish similar results for so-called weak induced arboricities and star arboricities of classes of graphs. 
\end{abstract}

\section{Introduction}
 For a graph $G$, the \emph{arboricity} $\arb(G)$ 
 is  the smallest number of forests
 covering all the edges of $G$. As Nash-Williams proved $50$ years ago, the arboricity is governed precisely by the largest density among the subgraphs of $G$. Let 
\[
 {\rm m}(G) = \max\left\{ \left\lceil \frac{|E(H)|}{|V(H)|-1} \right \rceil \colon H \subseteq G, |V(H)| \geq 2\right\}.
\]

\begin{theorem}[Nash-Williams~\cite{Nas-64}]\label{thm:Nash-Williams}
 For every graph $G$ with $|V(G)|\geq 2$ we have  $\arb(G) = {\rm m}(G)$.
\end{theorem}
 
This paper is concerned with the induced arboricity.
Recall that a subgraph $H$ of a graph $G=(V,E)$ is \emph{induced} if for any $u, v\in V(H)$, if $uv \in E(G)$  then $uv\in E(H)$.
We define a subgraph $H$ of $G$ to be \emph{weak induced} if each component of $H$ is an induced subgraph of $G$, but $H$ itself is not neccessarily an induced subgraph of $G$.

\begin{definition}
 The {\bf induced arboricity} $\iarb(G)$, respectively {\bf weak induced arboricity} $\wiarb(G)$, is the smallest number of induced, respectively weak induced forests, covering $E(G)$.
\end{definition}

For a complete graph $K_n$, $n\geq 2$, we see that $\arb(K_n)=\lceil n/2 \rceil$,   $\iarb(K_n) = \binom{n}{2}$ since there could be at most one edge in any induced forest of $K_n$, and
$\wiarb(K_n) =   n - 1 + (n \bmod 2)$ because any weak induced forest of $K_n$ is a matching.
Note that for the arboricity as well as the weak induced arboricity we can assume that an optimal set of forests covering the edges forms a partition of the edge set.
For the induced arboricity the smallest number of induced forests covering the edges might be smaller than the smallest number of induced forests partitioning the edge set of a graph.

\begin{proposition}\label{prop:cover-part}
 For each $k \geq 2$ there is a graph $G$ with $\iarb(G)=k$ and an edge $e \in E(G)$, such that in any cover of $E(G)$ with $k$ induced forests, $e$ is contained in all $k$ of them.
\end{proposition}

The notion of induced arboricity has been considered only recently, see for example a paper by some of the authors of the current paper, \cite{Axe-18}.
However, a special case of this parameter, when induced forests are required to be induced matchings, is a classical parameter, the strong chromatic index.
The \emph{strong chromatic index} $\chi'_s(G)$, introduced by Erd\H{o}s and Ne\v{s}et\v{r}il (c.f.~\cite{Fau-90}), is the smallest $k$ for which $E(G)$ can be covered with $k$ induced matchings.
In addition, the \emph{star arboricity} $\sarb(G)$ was considered extensively, where $\sarb(G)$ is the smallest number of star-forests, that is, forests with each component being a star, needed to cover the edges of $G$.

\medskip

For a graph parameter $p$ and a class of graphs $\cF$, let $p(\cF) = \sup\{p(G) \colon G\in \cF\}$.
We are concerned with induced arboricity $\iarb(\cF)$ and weak induced arboricity $\wiarb(\cF)$.
In addition, we consider the \emph{induced star arboricity} $\isarb(G)$ and \emph{weak induced star arboricity} $\isarb(G)$ of a graph $G$ defined as the smallest number of induced star-forests and weak induced star-forests, respectively, covering the edges of $G$.
Note that for the induced star arboricity as well as the weak induced star arboricity we can assume that an optimal set of forests covering the edges forms a partition of the edge set.

\medskip

Further note that every star-forest is a forest, an induced forest is also a weak induced forest, which in turn is also a forest.
Every matching is a weak induced star-forest and every induced matching is an induced star-forest.
Thus, denoting $\chi'(G)$ the edge-chromatic number of $G$, we have

\begin{gather}
  \arb(G) \leq \wiarb(G) \leq \iarb(G) \leq \isarb(G) \leq \chi'_s(G),\label{eq:basic:inequalities1}\\
  \wiarb(G)\leq \wisarb(G) \leq \isarb(G),\label{eq:basic:inequalities2}\\
  \arb(G)\leq \sarb(G) \leq \wisarb(G) \leq \chi'(G) \leq \chi'_s(G).\label{eq:basic:inequalities3}
\end{gather}

\begin{proposition}\label{prop:basic-inequalities}
 For any graph $G$ we have
 $\sarb(G) \leq 2\arb(G)$,
 $\wisarb(G) \leq 2 \wiarb(G)$, and
 $\isarb(G) \leq 3 \iarb(G)$.
\end{proposition}

This shows that star arboricities behave in a sense similar to their respective arboricities.
In particular Proposition~\ref{prop:basic-inequalities} implies that for any graph class $\cF$ we have
\begin{align*}
 \arb(\cF)\neq\infty &\Longleftrightarrow \sarb(\cF)\neq\infty,\\
 \iarb(\cF) \neq \infty &\Longleftrightarrow \isarb(\cF) \neq \infty,\\
 \wiarb(\cF) \neq \infty &\Longleftrightarrow \wisarb(\cF) \neq \infty.
\end{align*}
 
Next, we characterize exactly all graph classes $\cF$ for which $\iarb(\cF)\neq\infty$ and for which $\wiarb(\cF)\neq\infty$.
The characterization for induced arboricity is in terms of so-called shallow minors, a notion introduced by Ne{\v{s}}et{\v{r}}il and Ossona de Mendez~\cite{Nes-12}. 
A graph $H$ on vertex set $\{v_1,\ldots,v_n\}$ is a \emph{$\half$-shallow minor} of $G$ if there exist pairwise vertex-disjoint stars $S_1,\ldots,S_n$ in $G$ with centers $c_1,\ldots,c_n$ such that for each edge $v_iv_j$ in $H$ there is an edge in $G$ between $c_i$ and $S_j$, or between $c_j$ and $S_i$.
The set of all $\half$-shallow minors of $G$ is denoted $G \nabla \half$ and the set of all $\half$-shallow minors of graphs from a class $\cF$ is denoted $\cF \nabla \half$.

\begin{theorem}\label{thm:induced-for-classes}
 For any graph class $\cF$ we have
 \begin{align*}
   \iarb(\cF) \neq \infty
  &\Longleftrightarrow
  \chi(\cF \nabla \half) \neq \infty\\
  \text{and}\qquad\wiarb(\cF) \neq \infty
  &\Longleftrightarrow
  {\rm m}(\cF) \neq \infty
  \overset{\emph{\cite{Nas-64}}}{\Longleftrightarrow}
  \arb(\cF) \neq \infty.
 \end{align*}
\end{theorem}

The \emph{acyclic chromatic number}, $\chiacyc(G)$, of a graph $G$ is the smallest number of colors in a proper vertex-coloring such that the union of any two color classes forms an induced forest of $G$.
A result in~\cite{Axe-18} gives the following dependencies between induced arboricity and the acyclic chromatic number
\begin{equation}
 \log_3(\chiacyc(G))\leq \iarb(G) \leq \binom{\chiacyc(G)}{2}.\label{eq:iaChiAcyc}
\end{equation}
Thus, for any class $\cF$ of graphs we have $\iarb(\cF) \neq \infty$ if and only if $\chiacyc(\cF)\neq\infty$.
In turn Dvo\v{r}\'ak~\cite{Dvo-08} characterizes classes $\cF$ with $\chiacyc(\cF)\neq\infty$ in terms of shallow topological minors.
For a graph $H$ let $\sd_1(H)$ denote the graph obtained from $H$ by subdividing each edge exactly once.
Let $\SD(G)=\{H \colon \sd_1(H)\subseteq G\}$ and for a family $\cF$ of graphs let $\SD(\cF)=\bigcup_{G\in\cF}\SD(G)$.
Note that $\SD(G)\subseteq G\nabla\half$ and hence $\SD(\cF) \subseteq \cF\nabla\half$~\cite{Nes-12}.
Dvo\v{r}\'ak~\cite{Dvo-08} proved that $\chiacyc(\cF)\neq\infty$ if and only if $\chi(\SD(\cF)) \neq \infty$.
This gives the following corollary of  Theorem~\ref{thm:induced-for-classes} which is of independent interest.
\begin{corollary}
For any graph class $\cF$ we have \[\chi(\cF \nabla \half) \neq \infty\Longleftrightarrow \chi(\SD(\cF))\neq\infty.\]
\end{corollary}

Theorem~\ref{thm:induced-for-classes} shows that the weak induced arboricity is a parameter behaving in a sense like arboricity.
However, induced arboricity is more complex and depends on the structure of the graph, not only on the density.
In particular for the class of $d$-degenerate graphs we give the maximum values for arboricities and weak arboricities and show that the induced arboricities are unbounded.
For several classes $\cF$ of graphs with more restricted structure, that is, $\chi(\cF \nabla \half) \neq\infty$, we provide bounds on the various arboricities.
Specifically we consider graphs of given acyclic chromatic number, given tree-width, and the class of all planar graphs. 

The \emph{degeneracy} of a graph $G$ is the smallest integer~$d$ such that every subgraph of $G$ has a vertex of degree at most~$d$.
If the degeneracy of $G$ is~$d \geq 1$, then $\arb(G) \leq d < 2\arb(G)$~\cite{Burr86}.
That is, the degeneracy is closely related to the arboricity and hence the density of a graph.
For an integer $d \geq 1$, let $\cD_d$ denote the class of all graphs of degeneracy at most~$d$.
Theorem~\ref{thm:induced-for-classes} shows that $\arb(\cD_d)$, $\sarb(\cD_d)$, $\wiarb(\cD_d)$, and $\wisarb(\cD_d)$ are bounded while $\iarb(\cD_d) = \isarb(\cD_d) = \infty$.
 
\begin{theorem}\label{thm:degeneracy}
 For the class $\cD_d$ of all graphs of degeneracy~$d$, $d \geq 2$, we have
 \begin{align*}
  \arb(\cD_d) &= d,\\
  \sarb(\cD_d) &= 2d, \quad \emph{\cite{Alo-92}},\\
  \wiarb(\cD_d) = \wisarb(\cD_d) &= 2d,\\
  \iarb(\cD_d) = \isarb(\cD_d) &= \infty, \quad \emph{\cite{Axe-18}}.
 \end{align*}
\end{theorem}

Note that $\wisarb(\cD_d) \leq 2d$ is a strengthening of the result $\sarb(\cD_d) \leq 2d$ from~\cite{Alo-92}.
  
The \emph{tree-width} of a graph $G$ which is the smallest integer~$k$ such that $G$ is a subgraph of some chordal graph of clique number $k+1$~\cite{Rob-86}.
For a positive integer $k$, let $\cA_k$ denote the class of all graphs of acyclic chromatic number at most~$k$ and let $\cT_k$ denote the class of all graphs of tree-width at most~$k$. 
It is well-known that $\cT_{k-1} \subsetneq \cA_{k}$ for each $k \geq 2$, that is, graphs of tree-width~$k-1$ have acyclic chromatic number at most~$k$.
Inequalities~\eqref{eq:iaChiAcyc} state a functional dependency between induced arboricity and the acyclic chromatic number.
Here we establish upper bounds and specific values for the other arboricities on $\cA_k$.
Further we show that for each $k\geq 2$ these bounds are attained on the strict subclass $\cT_{k-1}\subsetneq \cA_k$, a result of independent interest.

\begin{theorem}\label{thm:acyclic-chromatic-number}
 Let $k \geq 2$, $\cF \in \{\cA_{k}, \cT_{k-1}\}$, and $\epsilon=0$ if $k$ is even and $\epsilon=1$ if $k$ is odd.
 Then  
 \begin{align*}
   \arb(\cF) &= k-1,\\
   \sarb(\cF) &= k, \quad \emph{\cite{Din-98,Duj-07,Hak-96}},\\
   \iarb(\cF) &= \binom{k}{2},\\
   \isarb(\cF) &= 3\binom{k}{2},\\
   \wiarb(\cF) &= k-1 + \epsilon.
 \end{align*}
 Moreover, $\wisarb(\cT_{k-1}) = 2k-2$ and $2k-2 \leq \wisarb(\cA_{k}) \leq 2k-2 + 2\epsilon$.
\end{theorem}

Finally we consider the family $\cP$ of planar graphs. Note that $\cP\subseteq \cA_5$~\cite{Bor-79} and $\cP\not\subseteq\cT_{k}$ for any $k\in\N$~\cite{Rob-2-86}.

\begin{theorem}\label{thm:planar}
 For the class $\cP$ of all planar graphs we have
 \begin{align*}
  \arb(\cP) &= 3,\\
  \sarb(\cP) &= 5, \quad \emph{\cite{Alg-89,Hak-96}},\\
  8 \leq \iarb(\cP) &\leq 10,\\
  18 \leq \isarb(\cP) &\leq 30,\\
  4\leq \wiarb(\cP) &\leq 5,\\
  6 \leq \wisarb(\cP) &\leq 10.
\end{align*}

For the class $\cO$ of all outerplanar graphs, we have 
\[
 \iarb(\cO) = 3, \quad \isarb(\cO) = 9, \quad \wiarb(\cO) = 3, \quad \wisarb(\cO) = 4.
\]
\end{theorem}
 
The bounds $8 \leq \iarb(\cP) \leq 10$ are also proved in~\cite{Doerr17}.
 
We prove Propositions~\ref{prop:cover-part} and~\ref{prop:basic-inequalities} and give the main construction in  Section~\ref{sec:general-observations}.
Theorems~\ref{thm:induced-for-classes}--\ref{thm:planar} are proved in the respective sections.

\section{Preliminary results}
\label{sec:general-observations} 

\subsubsection*{Proof of Proposition~\ref{prop:cover-part}}
 
 Let $k \geq 2$.
 Consider a complete bipartite graph $H = K_{k,k+1}$ with bipartition $(A,B)$, $|A| = k$, $|B| = k+1$.
 Every induced forest in $H$ is a star.
 Hence any induced forest $F$ in $H$ contains at most $\max(|A|,|B|) = |B| = k+1$ edges, implying that $\iarb(H) \geq |E(H)| / (k+1) = k$.
 Moreover, an induced forest $F$ has $k+1$ edges if and only if $F$ contains exactly one vertex of $A$ and all vertices of $B$, and taking one such forest $F$ for each vertex in $A$ gives the unique (up to isomorphism) cover of $E(H)$ with $|A| = k$ induced forests. Thus $\iarb(H)=k$.
 Note that every vertex in $B$ is contained in all $k$ forests.
 
 \begin{figure}[tb]
  \centering
  \includegraphics{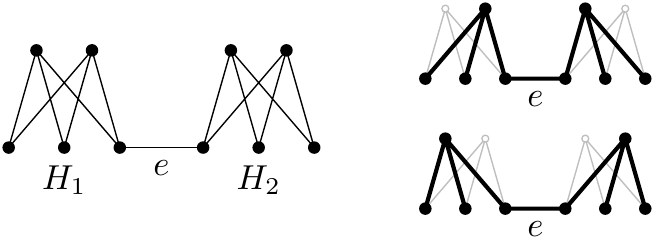}
  \caption{Left: A graph $G$ with an edge $e$ and $\iarb(G) = 2$, where in every cover of $E(G)$ with two induced forests $e$ is contained in both forests.
   Right: The unique (up to isomorphism) cover of $E(G)$ with two induced forests.}
 \label{fig:cover-vs-partition}
\end{figure}
 
 Now let $G$ be the graph obtained from two vertex-disjoint copies $H_1,H_2$ of $K_{k,k+1}$, together with an edge $e$ joining a vertex from the larger bipartition part of $H_1$ with  a vertex from the larger bipartition part of  $H_2$ .
 See Figure~\ref{fig:cover-vs-partition} for an illustration.
 Taking $k$ induced forests in $H_1$ covering $E(H_1)$ and $k$ induced forests in $H_2$ covering $E(H_2)$, we can pair these up and add $e$ to each of them to get $k$ induced forests in $G$ covering $E(G)$.
 Thus $k \geq \iarb(G) \geq \iarb(H_1) = k$.
 Observe that this gives the unique (up to isomorphism) cover of $E(G)$ with $k=\iarb(G)$ induced forests, and that $e$ is contained in all $k$ such forests.\qed

\subsubsection*{Proof of Proposition~\ref{prop:basic-inequalities}}
 
 For a set $\cF$ of forests, and each forest $F\in \cF$, fix a root of each tree in $F$ arbitrarily and 
 let $F_i$ be a subgraph of $F$ formes by the  edges  at distance~$i$ from the respective root in $F$.
 We refer to $F_i$'s as layers of $F$. 
 Then we see that each layer $F_i$ forms a star-forest and $F$ is the union of all $F_i$'s.
 
 \medskip
 
 Consider $\cF$ to be a smallest set of forests (respectively weak induced forests) covering $E(G)$.
 The forests formed by the even-indexed layers $\cup_{k\geq 0} F_{2k}$ or by the odd-indexed layers $\cup_{k\geq 0} F_{2k+1}$, $F\in \cF$, form star-forests (respectively weak induced star-forests) covering $E(G)$.
 This implies that $\sarb(G) \leq 2\arb(G)$ and $\wisarb(G) \leq 2 \wiarb(G)$.
 
 \medskip
  
 Let $\cF$ be a smallest set of induced forests covering $E(G)$.
 Consider each third layer of each forest and their unions:
 $\cup_{k\geq 0} F_{3k}$, $\cup_{k\geq 0} F_{3k+1}$, $\cup_{k\geq 0} F_{3k+2}$, $F\in \cF$.
 These $3|\cF|$ forests form induced star-forests covering $E(G)$.
 This implies that $\isarb(G) \leq 3 \iarb(G)$.\qed

\subsection{\boldmath Construction of the graph $G_k$}\label{subsec:Gk}
 For each $k\geq 2$ we shall construct a graph $G_k$ which is used in the proofs later.
 Let $G_2$ be the $10$-vertex tree in the left part of Figure~\ref{fig:G2-and-G3}.
 For $k \geq 3$ we shall define $G_k$ as follows.
 Let $H_1=H_1(k)$ be the $(k-1)^{\rm st}$ power of a path $P$ on $k(k-1)^2$ vertices, that is, $V(H_1)=V(P)$ and all pairs of vertices at distance at most~$k-1$ in $P$ are edges in $H_1$.
 Split the vertex set of $H_1$ into $k(k-1)/2$ pairwise disjoint sets of $2(k-1)$ consecutive vertices each.
 Let $\cS$ be the family of these sets. 
 For each $S_i\in \cS$, let $S_i=S_i'\cup S_i''$, where $S_i'$ consists of the first $k-2$ vertices and the $k^{\rm th}$ vertex in $S_i$ according to the order in $P$ and $S_i''=S_i-S_i'$.
 Note that each of $S_i'$ and $S_i''$ induces a clique on $k-1$ vertices in $H_1$.
 Add $k(k-1)$ new vertices, denoted $w_i'$, $w_i''$, $i=1,\ldots,k(k-1)/2$, and add all edges between $w_i'$ and $S_i'$ and between $w_i''$ and $S_i''$ for all $i$'s.
 Call the resulting graph $H_2=H_2(k)$.
 Finally, let $G_k$ be obtained as a union of $H_2$ and a set of $|V(H_2)|$ pairwise vertex-disjoint cliques $K_k$ that each  clique shares  exactly one vertex with $H_2$.
 That is, we ``hang'' $K_k$'s on the vertices of $H_2$ and refer to them as ``hanging cliques''.

 \begin{figure}[tb]
  \centering
  \includegraphics{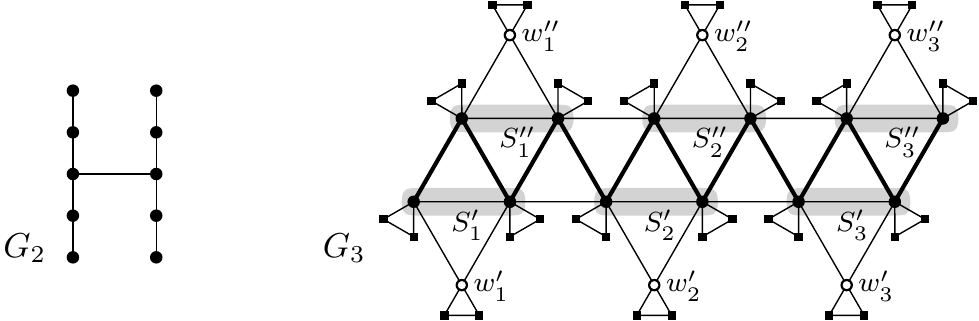}
  \caption{Two graphs $G_k$ of tree-width $k-1$ with $\wisarb(G_k) = 2(k-1)$ and $\isarb(G_k) = 3\binom{k}{2}$.
  Left: $k=2$ and $G_k$ has tree-width~$1$.
  Right: $k=3$ and $G_k$ has tree-width~$2$.
  }
  \label{fig:G2-and-G3}
 \end{figure}
 
 \begin{figure}[tb]
  \centering
  \includegraphics{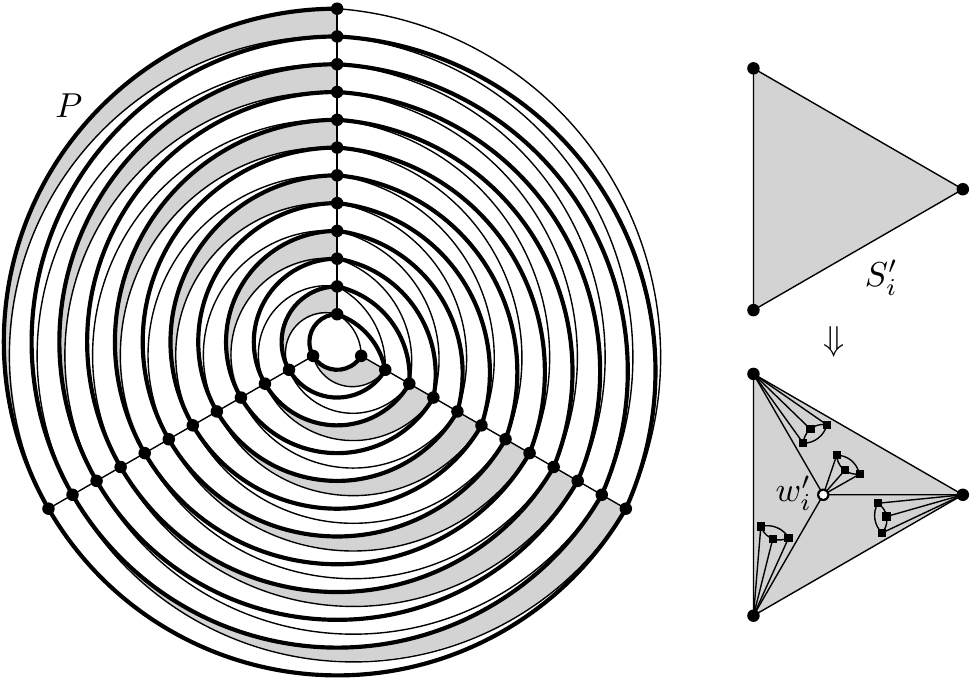}
  \caption{Illustration of the planar graph $G_4$ with $\wisarb(G_4) = 6$ and $\isarb(G_4) = 18$.
   Left: The graph $H_1 = H_1(4)$ is the fourth power of a $36$-vertex path $P$ (bold edges).
   The triangles induced by $S'_i$, $S''_i$ for $i = 1,\ldots,6$ are highlighted in gray and form a facial triangle in $H_1$.
   Right: $G_4$ is obtained from $H_1$ by augmenting each set $S'_i$, $S''_i$ with a vertex $w'_i$ respectively $w''_i$ and four ``hanging $K_4$'s''.
   }
  \label{fig:G4-round}
 \end{figure}

 See the right part of the Figure~\ref{fig:G2-and-G3} for an illustrative example when $k=3$ and Figure~\ref{fig:G4-round} for the case $k=4$.

 \begin{lemma}\label{lem:Gk}
  For any $k\geq 2$, the tree-width of $G_k$ is~$k-1$,
  $\isarb(G_{k}) \geq 3\binom{k}{2}$, and
  $\wisarb(G_{k}) \geq 2(k-1)$.
  Moreover $G_3$ is outerplanar and $G_4$ is planar.
 \end{lemma}
 \begin{proof}
 First observe that $G_2$ has tree-width~$1$ and satisfies $\chiacyc(G_2)=2$, $\isarb(G_2) = 3$, and $\wisarb(G_2) = 2$.
 
 Consider $k\geq 3$ and let $G=G_k$.
 As $H_1 = H_1(k)$ is a chordal graph of clique number~$k$, the tree-width of $H_1$ is~$k-1$.
 Moreover, $G$ is the union of $H_1$ and a family of edge-disjoint copies of $K_k$, each of which is chordal, has tree-width~$k-1$, and shares at most~$k-1$ vertices with $H_1$.
 It follows that also $G$ is chordal and has clique number~$k$.
 Therefore the tree-width of $G$ is~$k-1$.  
 
 \medskip
 
 Next we show that $\wisarb(G) \geq 2(k-1)$.
 To this end, consider a smallest set $F$ of forests of induced stars (not necessarily induced forests) partitioning the edge set of $G$.
 Without loss of generality we assume that each edge is contained in exactly one forest in $F$.
 If a star has only two vertices, we arbitrarily pick one vertex to be the center and the other to be the only leaf.
 We orient each edge $e$ in $G$ from the leaf to the center of the star containing $e$.
 Thus, $\wisarb(G)$ is at least the out-degree of any vertex. 
 The number of edges in the subgraph $H_1$ of $G$ is $k(k-1)^3 - \binom{k}{2}$.
 Each of these edges contributes~$1$ to the out-degree of some vertex in $H_1$.
 Thus, there is a set $S_i'$ or $S_i''$ (assume without loss of generality that it is $S_i'$) of total  out-degree at least $\lceil (k(k-1)^3 - \binom{k}{2})/(k(k-1))\rceil = (k-1)^2$ in $H_1$.
 Then there is a vertex $v\in S_i'$ with out-degree at least $(k-1)^2/|S_i'|=k-1$ in $H_1$.
 Thus there is a set $F'\subseteq F$ of at least~$k-1$ weak induced star-forests such that in each of the forests in $F'$ the vertex $v$ is a leaf of some star whose center is in $H_1$.
 
 On the other hand, consider the $k-1$ edges of $G$ incident to $v$ that are in the ``hanging clique'', that is, that are not in $H_2$.  
 These edges are in distinct forests from $F$ and neither of them is in $F'$, since such a forest would contain a $P_3$ whose endpoint is a center of a star.
 Thus $v$ is in at least $2(k-1)$ forests.
 This proves that $\wisarb(G) \geq 2(k-1)$.
 
 \medskip
 
 To prove that $\isarb(G) \geq 3\binom{k}{2}$ consider any smallest set $F$ of induced star-forests covering the edge set of $G$.
 Without loss of generality we assume that each edge is contained in exactly one forest in $F$.
 Again, we orient each edge $e$ in $G$ from the leaf to the center of the star containing $e$.
 As argued above, without loss of generality, there is a set $S_i'$ of total out-degree at least $(k-1)^2$ in $H_1$.
 By the choice of $S_i'$, at least $(k-1)^2$ stars in $F$ have a leaf in $S_i'$ and the center in $V(H_1)$.
 At most $\binom{k-1}{2}$ such stars have their center in $S_i'$ because there are $\binom{k-1}{2}$ edges in $S_i'$.
 Thus there is a set $X$ of at least $(k-1)^2 - \binom{k-1}{2} = \binom{k}{2}$ stars in $F$ that have a leaf in $S_i'$ and the center in $V(H_1) - S_i'$.
 Since $S_i'$ induces a clique, the stars from $X$ belong to distinct forests of $F$.
 
 Let $S= S_i'\cup \{w_i'\}$.
 Each vertex in $S$ is incident to $k-1$ edges of a ``hanging clique'' which belong to distinct forests of $F$ because any two such edges either induce a triangle in the respective ``hanging clique'' or have an edge of $G[S]$ between their endpoints.
 Call the set of these $k(k-1)$ forests $Y\subseteq F$.
 We see that $X$ and $Y$ are disjoint. 
 Thus $\isarb(G_t) \geq |X| + |Y| = \binom{k}{2} + 2\binom{k}{2} = 3\binom{k}{2}$.

 \medskip

 The facts that $G_3$ is outerplanar and $G_4$ is planar follow from the embeddings given in Figure~\ref{fig:G2-and-G3} and Figure~\ref{fig:G4-round}, respectively.
\end{proof}

\section{Proof of Theorem~\ref{thm:induced-for-classes}}\label{sec:minorProof}

First suppose that $\chi(\cF\nabla\half)\neq\infty$.
We shall prove that $\iarb(\cF)\neq\infty$.
We have $\chi(\cF)\neq\infty$ and $\chi(SD(G))\neq\infty$ since $\cF$, $SD(\cF)\subseteq \cF\nabla\half$.
Thus $\chiacyc(\cF)\neq\infty$ due to Dvo\v{r}\'{a}k~\cite{Dvo-08}.
Hence $\iarb(\cF)\neq\infty$ as $\iarb(G)\leq\chiacyc(G)^2$ for any graph $G$ by inequality~\eqref{eq:iaChiAcyc}.

\medskip

Now suppose that $\iarb(\cF)\neq\infty$.
Then $\chi(\cF\nabla\half)\neq\infty$ holds due to the following claim.

\begin{claim}\label{claim:nabla-half-at-most-iarb}
 For every graph $G$ we have $\chi(G \nabla \half) \leq \iarb(G)2^{3\iarb(G)+1}$.
\end{claim}

 Let $\varphi$ be a proper vertex coloring of $G$ with at most $2\iarb(G)$ colors.
 This exists as $\arb(G) \leq \iarb(G)$ and hence $G$ is $(2\iarb(G)-1)$-degenerate.
 Moreover, let $F_1,\ldots,F_k$, $k \leq 3\iarb(G)$ be induced star-forests covering $E(G)$, which exist as $\isarb(G) \leq 3\iarb(G)$.
 
 Now consider $H$ with vertex set $v_1,\ldots,v_n$ to be an arbitrary but fixed $\half$-shallow minor of $G$.
 Let $H$ be defined by vertex-disjoint stars $S_1,\ldots,S_n$ in $G$ with centers $c_1,\ldots,c_n$, respectively.
 Let $E(H) = E_1 \dot\cup E_2$, where $E_1$ contains all edges $v_iv_j$ for which $c_ic_j \in E(G)$ and $E_2 = E(H) - E_1$ contains all remaining edges $v_iv_j$ for which $c_i$ is adjacent to a leaf in $S_j$, or $c_j$ is adjacent to a leaf in $S_i$.

 We define a vertex coloring $\psi$ of $H$ as follows:
 For $i=1,\ldots,n$ let $A_i = \{j \in [k] \colon E(F_j) \cap E(S_i) \neq \emptyset\}$ be the indices of those star-forests that contain at least one edge in $S_i$.
 For each vertex $v_i$ define the color of $v_i$ to be $\psi(v_i) = (\varphi(c_i),A_i)$.
 Clearly, $\psi$ uses at most $2\iarb(G)2^k \leq \iarb(G)2^{3\iarb(G)+1}$ colors.
 
 \begin{figure}
  \centering
  \includegraphics{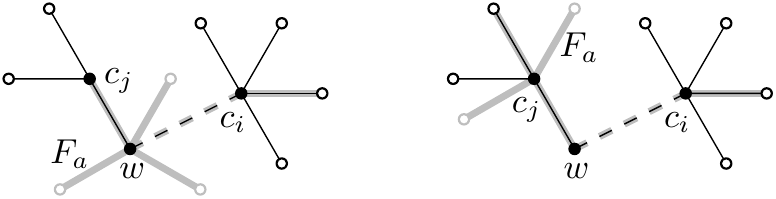}
  \caption{Two stars $S_i$ and $S_j$ (black edges) with centers $c_i$ and $c_j$, an edge (dashed) between a leaf $w$ of $S_j$ and the center $c_i$ of $S_i$. If an induced subgraph $F_a$ (bold gray edges) contains $wc_j$ and an edge in $S_i$, then it also contains $wc_i$.}
  \label{fig:f-graph-classes}
 \end{figure}

 If $v_iv_j$ is an edge in $E_1$, then $\varphi(v_i) \neq \varphi(v_j)$ and hence $\psi(v_i) \neq \psi(v_j)$.
 Assume that  $v_iv_j$ is an edge in $E_2$, say center $c_i$ is adjacent to a leaf $w$ in $S_j$, consider the index $a \in A_j$ such that edge $wc_j$ is in induced star-forest $F_a$.

 Then $c_i \notin V(F_a)$ since otherwise all the edges $wc_j$, $wc_i$ and some edge from $S_i$ are induced by $V(F_a)$ in $G$ and contradict that $F_a$ is an induced star-forest.
 See Figure~\ref{fig:f-graph-classes} for an illustration.
 Therefore $S_j \neq S_i$ and thus $\psi(v_i) \neq \psi(v_j)$.
 So $\psi$ is indeed a proper vertex coloring of the $\half$-shallow minor $H$ of $G$, proving that $\chi(G \nabla \half) \leq \iarb(G)2^{3\iarb(G)+1}$, which concludes the proof of Claim~\ref{claim:nabla-half-at-most-iarb}.
 
 \medskip

 Next we prove that $\wiarb(\cF)\neq\infty$ if and only if $\arb(\cF)\neq\infty$.
 For each graph $G$ we have $\arb(G)\leq\wiarb(G)$ by inequality~\eqref{eq:basic:inequalities1}.
 Thus $\wiarb(\cF)\neq\infty$ implies $\arb(\cF)\neq\infty$.
 So assume that $\arb(\cF)\neq\infty$.
 Then $\wiarb(\cF)\neq\infty$ holds due to the follow claim.
 
 \begin{claim}\label{claim:wia-by-a}
  For every graph $G$ we have $\wiarb(G)\leq4\arb(G)^2$.
 \end{claim}

  Let $\cF=\{S_1,\ldots,S_{\sarb(G)}\}$ be a smallest set of star-forests covering $E(G)$ and let $c$ be an optimal proper coloring of $V(G)$.
 For all $i\in[\sarb(G)]$ and $j\in[\chi(G)]$ let $F_{i,j}$ denote the subgraph of $S_i$ formed by all edges whose leaf vertex is colored $j$ under $c$.
 Then $F_{i,j}$ is a weak induced star-forest.
 This shows that $\wiarb(G)\leq \sarb(G)\chi(G)\leq 2\arb(G)2\arb(G)\leq 4\arb(G)^2$.
 Here $\sarb(G)\leq 2\arb(G)$ by Proposition~\ref{prop:basic-inequalities} and $\chi(G)\leq 2\arb(G)$ as each graph $G$ is $(2\arb(G)-1)$-degenerate~\cite{Burr86}.
 This concludes the proof of Claim~\ref{claim:wia-by-a} and the proof of Theorem~\ref{thm:induced-for-classes}.

\section{Proof of Theorem~\ref{thm:degeneracy}}
\label{sec:degeneracy}
Let $d\geq 2$.
The fact that $\arb(\cD_d) \leq d$ follows from Nash-Williams' Theorem.
For the lower bound $\arb(\cD_d) \geq d$ it suffices to observe that the complete bipartite graph $K_{n,d}$, with $n>(d-1)^2$ satisfies $K_{n,d} \in \cD_d$ and $\arb(K_{n,d}) = d$ by Nash-Williams' Theorem.

Some of the authors of this paper constructed  $2$-degenerate graphs of arbitrarily large induced arboricity, see \cite{Axe-18}, showing that $\iarb(\cD_d) = \isarb(\cD_d) = \infty$.
We shall show first that $\wisarb(G) \leq 2d$, for any $d$-degenerate graph $G$.

\medskip

An ordering $v_1<\cdots<v_n$ of the vertices of $G$ is a \emph{$d$-degeneracy ordering} if for each $i\in[n]$ the vertex $v_i$ has at most~$d$ neighbors $v_j$ with $j<i$.
We think of identifying vertices with some points on a horizontal line.
We say that a vertex $v_i$ lies to the right (left) of $v_j$ if $i>j$ ($i<j$).
A star is a \emph{right (left) star} with respect to the vertex ordering if its center lies left of (right of) all its leaves.
We shall prove the following claim by induction on $n = |V(G)|$.

\begin{claim*}
 For any $d$-degenerate graph $G$ and any $d$-degeneracy ordering $v_1<\cdots<v_n$ of $G$, there is a coloring $c$ of $E(G)$ in colors $\{1,\ldots,2d\}$ and sets of colors $S(v_i)\subseteq[2d]$, $i=1,\ldots,n$, such that each color class is a weak induced star-forest of right stars, $|S(v_i)|=d$, and $S(v_i)$ does not contain any color from $c(v_iv_j)$ for $v_iv_j\in E(G)$ and $j<i$, i.e., of the edges going to the left from $v_i$.
\end{claim*}
When $|V(G)|=1$, let $S(v_1)=[d]$.
All conditions are clearly satisfied. 
Assume that $n=|V(G)|>1$ and consider a $d$-degeneracy ordering $v_1<\cdots<v_n$ of $G$.
Let $G'=G - v_n$.
Then $v_1<\cdots<v_{n-1} $ is a $d$-degeneracy ordering of $G'$.
Consider a coloring $c$ of $E(G')$ guaranteed by the induction hypothesis and respective color sets $S(v_1),\ldots,S(v_{n-1})$.
Let $w_1=v_{i_1}, \ldots, w_t= v_{i_t}$ be the neighbors of $v_n$ in $G$ with $i_1<\ldots <i_t$, $t\leq d$.
Further let $S'(w_i)= S(w_i)- \{ c(w_iw_j) \colon w_iw_j\in E(G), i<j\}$, i.e., the sets $S'(w_i)$ are obtained from $S(w_i)$ by deleting the colors of the edges going from $w_i$ to the right, to another neighbor of $v_n$.
Then $|S'(w_i)|\geq d-(t-i)\geq i$ for each $i=1,\ldots,t$.
Now we shall extend the coloring $c$ to a coloring of $|E(G)|$, such that $c(v_nw_1)\in S'(w_1)$, $c(v_nw_2)\in S'(w_2)-\{c(v_nw_1)\}$, and so on, $c(v_nw_i) \in S'(w_i)-\{c(v_nw_1),\ldots,c(v_nw_{i-1})\}$, $i=2,\ldots,t$.
This procedure is possible since $|S'(w_i)|\geq i$, $i=1,\ldots,t$.
Let $S(v_n) \subseteq [2d]- \{c(v_nw_1),\ldots,c(v_nw_t)\}$ such that $|S(v_n)|=d$.
We shall prove that this coloring and the color sets $S(v_1),\ldots,S(v_{n})$ satisfy the Claim.

Within $G'$ each color class is a union of right stars by induction.
Moreover all monochromatic (left) stars centered at $v_n$ are single edges $v_nw_i$ whose color is in $S(w_i)$, that is, their color is not assigned to any edge going from $w_i$ to the left.
Hence all color classes are unions of right stars.
The sets $S(v_1),\ldots,S(v_n)$ satisfy the conditions of the Claim.
We only need to verify that each color class is a weak induced star-forest.
Since each color class in $G'$ is a weak induced star-forest, it is sufficient to consider monochromatic stars containing an edge $v_nw_i$ for some $i \in [t]$.
The color of the edge $v_nw_i$ is in $S'(w_i)$.
Thus neither of the neighbors $w_j$ of $v_n$ with $j\neq i$ is a leaf in the monochromatic (right) star containing $v_nw_i$.
Hence this star is induced in $G$.
This proves the Claim.

\medskip
 
 The claim in particular shows that $\wiarb(\cD_d) \leq \wisarb(\cD_d) \leq 2d$.
 Next, we shall show that for every $d \geq 2$ there is a $d$-degenerate bipartite graph $G$ with $\wiarb(G) \geq 2d$.

 \medskip
 
 Let $d \geq 2$ be fixed and let $N = 2^dd^{d+1}$.
 Consider pairwise disjoint sets $A$, $B$, and $B_S$, for each $S\in \binom{B}{d}$, where $|A|=d$ and $|B| = |B_S|=N$, $S\in \binom{B}{d}$.
 Let $G$ be a union of complete bipartite graphs with parts $(A, B)$ and $(S, B_S)$, $S\in \binom{B}{d}$.
 It is easy to see that the resulting graph $G$ is $d$-degenerate and bipartite.
 So assume that $E(G)$ is covered by some number $x<2d$ of weak induced forests.
 Without loss of generality assume that each edge is contained in exactly one of these forests.
 Consider a coloring $c:E(G)\rightarrow \{1, \ldots, x\}$ such that $c(e)=i$ if $e$ is in the $i^{\rm th}$ forest.
 Let $A=\{a_1, \ldots, a_d\}$.
 For each vertex $b \in B$, let $\bar{c}(b) = (c(ba_1), c(ba_2), \ldots, c(ba_d))$.
 We see that there are $|B|= 2^dd^{d+1} = d \cdot (2d)^d \geq d \cdot x^d$ such tuples $\bar{c}(b)$, $b\in B$.
 The total possible number of such distinct tuples is $x^d$.
 Thus, by pigeonhole principle, there is a set $S\subseteq B$ such that for any two elements $b$, $b'\in S$ we have $\bar{c}(b)=\bar{c}(b')$.
 This implies that under coloring $c$ the subgraph $G[(A, S)]$ is a union of monochromatic stars, each with center in $A$ and leaf-set $S$.
 Since each color class is a forest, all these stars have different colors.
 So in total there are at least~$d$ colors.
 By a similar argument, we see that there is a subset $T$ of $B_S$ so that $|T|=d$,   $G[(S, T)]$ is a union of $d$ monochromatic stars with centers in $S$ and leaf-sets $T$.  Note that a star with center $a\in A$ and leaf-set $S$ and a star with center $s\in S$ and leaf-set $T$ together do not induce a forest,
 so they must be of different colors. Thus the total number of colors is at least $2d$, a contradiction.
 Thus $\wiarb(\cD_d) \geq 2d$.
 This concludes the proof of Theorem~\ref{thm:degeneracy}.

\section{Proof of Theorem~\ref{thm:acyclic-chromatic-number}}

First we shall establish the upper bounds.
Recall that $\cT_{k-1}\subseteq A_{k}$.
So let $G\in \cA_k$ be an $n$-vertex graph.
First we consider the arboricity of $G$.
Consider an acyclic coloring of $G$ with $k$ colors and let $a_1, \ldots, a_k$ be the sizes of the color classes.
Then $a_1+\cdots + a_k =n$ and since each edge is induced by some two color classes and each pair of color classes induces a forest, we have that 
\[|E(G)|  \leq \sum_{1\leq i<j\leq k} (a_i+a_j -1) = (k-1)\sum_{1\leq i\leq k} a_i  - \binom{k}{2} = (k-1)n-\binom{k}{2}.\]
Since each subgraph of $G$ on $n'$ vertices has acyclic chromatic number at most $k$, it has at most $(k-1)n'-\binom{k}{2}$ edges.
Thus ${\rm m}(G) \leq k-1$ and by Nash-Williams' Theorem (Theorem~\ref{thm:Nash-Williams}) $\arb(\cT_{k-1})\leq \arb(\cA_k) \leq k-1$ for any $k \geq 2$.

\medskip

The upper bound on the star arboricity follows from a result of Hakimi~\textit{et al.}~\cite{Hak-96} showing that $\sarb(G)\leq \chiacyc(G)$.
Hence $\sarb(\cT_{k-1})  \leq \sarb(\cA_k) \leq k$ for any $k \geq 2$.

\medskip
 
Next we shall show that $\wiarb(G) \leq k-1 + (k \bmod 2)$.
Consider an acyclic coloring of $V(G)$ with $k$ color classes $V_1,\ldots,V_k$.
Consider a complete graph $H$ on vertex set $\{V_1,\ldots,V_k\}$ and a partition of $E(H)$ into $m$ matchings $M_1,\ldots,M_m$, where $m = \chi'(K_k) = k-1 + (k \bmod 2)$.
For $\ell = 1,\ldots,m$ let $G_\ell$ be the spanning subgraph of $G$ consisting of those edges $xy$ with $x \in V_i$, $y \in V_j$, and $V_iV_j \in E(M_\ell)$.
It is easy to see that $G_1,\ldots,G_m$ is a partition of $E(G)$ into weak induced forests, which implies that  $\wiarb(G) \leq k-1 + (k \bmod 2)$.
Thus $\wiarb(\cT_{k-1})  \leq \wiarb(\cA_k)\leq k-1 + (k \bmod 2)$.
 
\medskip
 
We have $\iarb(\cA_k) \leq \binom{k}{2}$ by inequality~\eqref{eq:iaChiAcyc}.
The remaining upper bounds follow from Proposition~\ref{prop:basic-inequalities}, which gives $\isarb(\cT_{k-1})  \leq \isarb(\cA_k) \leq 3\iarb(\cA_k) \leq 3\binom{k}{2}$ and $\wisarb(\cA_k) \leq 2\wiarb(\cA_k) \leq 2k-2+2(k \bmod 2)$.
Moreover $\wisarb(\cT_{k-1})\leq 2k-2$ follows from Theorem~\ref{thm:degeneracy} since each graph in $\cT_{k-1}$ is $(k-1)$-degenerate.

\medskip

To see the lower bounds, note first that the complete graph $K_k$ has tree-width $k-1$, $\iarb(K_k)=\binom{k}{2}$, and $\wiarb(K_k)=\chi'(K_k)=k-1+(k \bmod 2)$.
Moreover Dujmovi\'c and Wood~\cite{Duj-07} show that $\arb(\cA_k)\geq \arb(\cT_{k-1}) \geq k-1$ and $\sarb(\cA_k)\geq \sarb(\cT_{k-1}) \geq k$.
For the remaining lower bounds consider the graph $G_k$ constructed in Section~\ref{subsec:Gk}.
By Lemma~\ref{lem:Gk}, for any $k \geq 2$, we have $G_k\in\cT_{k-1}$, $\wisarb(G_k) \geq 2k-2$, and $\isarb(G_k) \geq 3 \binom{k}{2}$.
This concludes the proof of Theorem~\ref{thm:acyclic-chromatic-number}.

\section{Proof of Theorem~\ref{thm:planar}}

Consider the class $\cP$ of all planar graphs.
Note that $\arb(\cP)=3$ follows from Nash-Williams' Theorem.

 The upper bound $\iarb(\cP) \leq 10$ follows from Borodin's result $\chiacyc(\cP) \leq 5$ in~\cite{Bor-79} and Proposition~\ref{prop:basic-inequalities} as already observed in~\cite{Axe-18}.
 Proposition~\ref{prop:basic-inequalities} gives $\isarb(\cP) \leq 3\iarb(\cP) \leq 30$.
 Similarly, $\chiacyc(\cP) \leq 5$ and Theorem~\ref{thm:acyclic-chromatic-number} immediately give $\wiarb(\cP) \leq 5$.
 Proposition~\ref{prop:basic-inequalities} then gives $\wisarb(\cP) \leq 2 \wiarb(\cP) \leq 10$.
 The lower bounds $\isarb(\cP) \geq \isarb(G_4) \geq 18$ and $\wisarb(\cP) \geq \wisarb(G_4) \geq 6$ follow from Lemma~\ref{lem:Gk} for the planar graph $G_4$.
 
 \medskip

In order to prove the remaining lower bounds $\wiarb(\cP) \geq 4$ and $\iarb(\cP) \geq 8$, we shall use the following preliminary observations.
For an integer $\ell \geq 3$, the double-wheel graph $DW_\ell$ consists of an $\ell$-cycle $C_\ell$, that we refer to as the \emph{rim} of the double-wheel, and two additional vertices $x,y$, called \emph{hubs}, with $N(x) = N(y) = V(C_\ell)$.
Note that $DW_\ell$ is planar for any $\ell \geq 3$.

\begin{lemma}\label{lem:double-wheels}
 Let $\ell$ be odd and consider $k$ induced forests covering $E(DW_\ell)$.
 \begin{enumerate}[label = (\roman*)]
  \item If $\ell \geq 5$, then $k \geq 7$ and some vertex is contained in at least four forests.\label{enum:strong-DW5}
  \item If $\ell \geq 7$ and a hub is contained in at least four forests, then $k \geq 8$.\label{enum:strong-DW7}
 \end{enumerate}
\end{lemma}
\begin{proof}
  Let $G = DW_\ell$ be covered with a set $T$ of $k$ induced forests.
  Let $x$ and $y$ be the hub vertices of $G$ and $C$ be the cycle of $G$.
  For a hub vertex $v\in \{x, y\}$ in $G$ let $A_v$, respectively $B_v$, be the set of forests $F \in T$ with $v \in V(F)$ that contain at most one edge, respectively at least two edges, incident to $v$.
  Let $S \subseteq T$ be the set of forests that contain at least one edge of $C$.
  Note that among $A_x$, $A_y$, $B_x$, $B_y$, and $S$ only $A_x$ and $A_y$ may have a non-empty intersection.

  \begin{enumerate}[label = (\roman*)]
   \item Let $\ell \geq 5$.    
    Clearly, $|S| \geq 2$.
    Since $\ell$ is odd, each forest $F \in B_v$ contains at most $(\ell-1)/2$ edges incident to $v$.
    This implies that $|A_v| + 2|B_v| \geq 5$.
    Indeed, this clearly holds if $|B_v|\geq 3$.
    Further, if $|B_v| = 0$, then $|A_v| \geq \ell \geq 5$.
    If $|B_v| = 1$, then $|A_v| \geq (\ell+1)/2 \geq 3$.
    Finally, if $|B_v| = 2$, then $|A_v| \geq 1$.
    Assuming without loss of generality that $|B_x| \leq |B_y|$, we conclude that $|A_x \cup B_x \cup A_y \cup B_y| \geq |A_x| + |B_x| + |B_y| \geq 5 - |B_x| + |B_y| \geq 5$.
    It follows that $k = |T| \geq |S| + |A_x \cup B_x \cup A_y \cup B_y| \geq 2+5 = 7$.

    Moreover, note that if a vertex $z \in V(C)$ is contained in only one forest $F \in A_x \cup B_x \cup A_y \cup B_y$, then $F \in A_x \cap A_y$.
    If each of $x,y$ is contained in at most three forests of $T$, then $|A_x|, |A_y| \leq 1$ and $|B_x|, |B_y| \geq 2$ (since $\ell\geq 5$ is odd), and thus at most one vertex in $C$ is contained in only one forest from $A_x \cup B_x \cup A_y \cup B_y$.
    However, at least two vertices on $C$ are contained in two forests of $S$ and it follows that in this case at least one of them is contained in two forests of $S$ and two forests of $T-S$, as desired.
   
   \item Let $\ell \geq 7$ and assume that  a hub vertex $x$ is contained in at least four forests of $T$.
    If $|B_y| = 0$, then $|A_y| \geq \ell \geq 7$ and thus $k= |T| \geq |A_y| + |S| \geq 7+2 = 9$.
    So we may assume that $|B_y| \geq 1$ and symmetrically $|B_x| \geq 1$.
    If $|B_y| = 1$, then $|A_y| \geq (\ell+1)/2 \geq 4$ and thus $k= |T| \geq |B_x| + |B_y| + |A_y| + |S| \geq 1 + 1 + 4 + 2 = 8$.
    And finally if $|B_y| \geq 2$, then $k= |T| \geq |B_y| + |A_x \cup B_x| + |S| \geq 2 + 4 + 2 = 8$, which proves the claim.\qedhere
 \end{enumerate}
\end{proof}

 Finally, it remains to prove $\wiarb(\cP) \geq 4$ and $\iarb(\cP) \geq 8$.
 To this end construct a graph $G$ from one copy $X$ of $DW_5$ on a vertex set $\{v_1,\ldots,v_7\}$ and seven copies $X_1,\ldots,X_7$ of $DW_7$ by identifying for $i =1,\ldots,7$ a vertex $v_i$ in $X$ with a hub vertex of $X_i$, see Figure~\ref{fig:LB-8-planar}.
 \begin{figure}[tb]
 \centering
 \includegraphics[angle = 90]{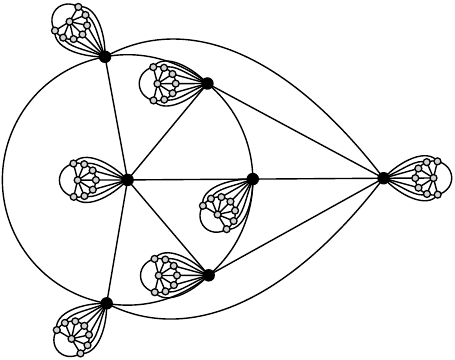}
 \caption{A planar graph $G$ with $\wiarb(G) \geq 4$ and $\iarb(G) = 8$.}
 \label{fig:LB-8-planar}
\end{figure}
 Note that $G$ is planar.
 To see that $\wiarb(G) \geq 4$, consider a set $F$ of weak induced forests covering $E(G)$.
 The components of these forests are induced trees which together cover $E(G)$.
 So Lemma~\ref{lem:double-wheels}~\ref{enum:strong-DW5} implies that some vertex is contained in at least four different components.
 As these components belong to different forests from $F$ we have $|F|\geq 4$.
 Next assume that $E(G)$ is covered with a set $T$ of induced forests.
 By Lemma~\ref{lem:double-wheels}~\ref{enum:strong-DW5}, some vertex $v_i$ of $X$ is contained in at least four forests of $T$.
 By Lemma~\ref{lem:double-wheels}~\ref{enum:strong-DW7}, the corresponding copy $X_i$ of $DW_7$ contains at least eight forests of $T$, proving that $|T| \geq 8$, as desired.

\medskip

Now we shall consider outerplanar graphs.
Since outerplanar graphs have tree-width at most~$2$, $\iarb(\cO)\leq 3$, $\wiarb(\cO)\leq 3$,  $\isarb(\cO)\leq 9$, and $\wisarb(\cO)\leq 4$,  as follows  from Theorem~\ref{thm:acyclic-chromatic-number}.
On the other hand $\iarb(K_3)=3$ and $\wiarb(K_3)=3$, thus $\iarb(\cO)=\wiarb(\cO)=3$. 
Lemma~\ref{lem:Gk} implies that $\isarb(G_3) \geq 9$ and $\wisarb(G_3)\geq 6$, thus, since $G_3$ is outerplanar, 
$\isarb(\cO)=9$ and $\wisarb(\cO)=6$.
This concludes the proof of Theorem~\ref{thm:planar}.

\bibliographystyle{abbrv}
\bibliography{lit}
\end{document}